\documentclass{amsart}
\usepackage{geometry, comment}  
\usepackage{parskip}  
\usepackage{graphicx}
\usepackage{amssymb,amsmath}
\usepackage{hyperref}
\usepackage{pdfsync}
\usepackage{footmisc}
\usepackage[all]{xy}

\newtheorem{thm}{Theorem}[section]
\newtheorem{prop}[thm]{Proposition}
\newtheorem{lem}[thm]{Lemma}

\newtheorem*{thm*}{Theorem}

\theoremstyle{definition}
\newtheorem{defn}{Definition}
\newtheorem{conj}{Conjecture}

\newtheorem{rem}[thm]{Remark}

\newtheorem*{defn*}{Definition}
\newtheorem*{rem*}{Remark}

\DeclareMathOperator{\Gr}{Gr}
\DeclareMathOperator{\ncl}{ncl}
\DeclareMathOperator{\rad}{rad}
\DeclareMathOperator{\im}{im}

\newcommand{\BN}{\mathbb{N}}

\newcommand{\BZ}{\mathbb{Z}}
\newcommand{\FF}{\mathbb{F}}
\newcommand{\factor}[2]{{\raise0.7ex\hbox{$#1$} \!\mathord{\left/ {\vphantom {#1 {#2}}}\right.\kern-\nulldelimiterspace}\!\lower0.7ex\hbox{${#2}$}}}

\author[M. Casals-Ruiz]{Montserrat Casals-Ruiz}
\address{Matematika Saila, UPV/EHU, Sarriena s/n, 48940, Leioa - Bizkaia, Spain}
\email{montsecasals@gmail.com}

\author[I. Kazachkov]{Ilya Kazachkov}
\address{Ikerbasque - Basque Foundation for Science and Matematika Saila,  UPV/EHU,  Sarriena s/n, 48940, Leioa - Bizkaia, Spain}
\email{ilya.kazachkov@gmail.com}

\author[V. Remeslennikov]{Vladimir Remeslennikov}
\address{Sobolev Institute of Mathematics (Russian Academy of Sciences), 13 Pevtsova St., Omsk 644099, Russia}
\email{remesl@ofim.oscsbras.ru}
\thanks{The first author is supported by the Juan de la Cierva Programme of the Spanish Government. The second author is supported by the ERC grant PCG-336983. The first two authors are partly supported by the Spanish Government, grant MTM2014-53810-C2-2-P. The third author is supported by the grant of the Russian Fund for Basic Research N14-01-00068.}

\begin{document}
\title{Pro-Hall $R$-groups and groups discriminated by the free pro-$p$ group.}

\begin{abstract}
In this note we introduce pro-Hall $R$-groups as inverse limits of Hall $R$-groups and show that for the binomial closure $S^{bin}$ of any ring $S$ discriminated by $\BZ_p$, the free pro-Hall $S^{bin}$-group $\FF(A,S^{bin})$ is fully residually free pro-$p$. Furthermore, we prove that any finite set of elements in $\FF(A,S^{bin})$ defines a pro-$p$ subgroup and so an irreducible coordinate group over the free pro-$p$ group.
\end{abstract}

\dedicatory{To the memory of Oleg Melnikov.}

\maketitle

\section{Introduction}

The class of limit groups arises naturally in many different aspects of the theory of free groups and is fundamental in their model-theoretic study. Its versatile nature is reflected in the rich characterisation that this class admits: from the viewpoint of model theory, they are finitely generated models of the universal theory of a free group; from the group-theoretic perspective, limit groups are simply finitely generated fully residually free groups; in universal algebra, they are finitely generated subgroups of the ultrapower of a free group; in algebra-geometric terms, they are coordinate groups of irreducible algebraic sets over a free group. We refer the interested reader to \cite{DMR1} for a detailed discussion of these and some of the other characterisations of the class of limit algebras over an algebraic structure.

\smallskip

In 1960, Lyndon while studying one variable equations over the free group, introduced the exponential group $F^{\BZ[t]}$ and showed that this group is fully residually free, \cite{Lyndon}. 
Hence, the group $F^{\BZ[t]}$ provided a wealth of example of limit groups and was the first tool for their systematic study. In 1996, Miasnikov and Remeslennikov gave a description of the exponential group as iterated sequence of free extensions of centralisers. This description uncovered a very robust algebraic structure of the finitely generated subgroups of the exponential group:  they are built from free abelian groups of finite rank by taking free products, and amalgamated products and separated HNN-extensions with abelian amalgamated subgroups. Furthermore, in the same paper, the authors conjectured that the exponential group $F^{\BZ[t]}$ was a universe for all limit groups over a free group or, in other words, that any finitely generated fully residually free group was a subgroup of $F^{\BZ[t]}$. This conjecture was proven in 1998 by Kharlampovich and Miasnikov, see \cite{KM2}, and later, using different methods, by Sela, and by Champetier and Guirardel.

\smallskip

In \cite{KZ}, Kochloukova and Zalesskii begun a study of pro-$p$-analogues of limit groups. Since the usual combinatorial and geometric techniques can not be applied for pro-$p$ groups, the authors approach pro-$p$ analogues of limit groups via their algebraic characterisation as subgroups of iterated sequences of extensions of centralisers. More concretely, they define the class of pro-$p$ groups $\mathcal L$ as the class that contains finitely generated pro-$p$ subgroups of sequences of extensions of centralisers performed in the category of pro-$p$ groups starting with a finitely generated free pro-$p$ group. However, as the authors point out, it is not known whether the class of pro-$p$ groups $\mathcal L$ contains \emph{only} fully residually free pro-$p$ groups and whether it contains \emph{all} finitely generated fully residually free pro-$p$ groups.

\smallskip

The main goal of this note is to present an alternative approach to the study of limit groups over free pro-$p$ groups via an analogue of Lyndon's exponential group. More precisely, we introduce a new class of groups $\FF(A,R)$, called the \emph{free pro-Hall $R$-groups} as inverse limits of Hall $R$-groups and show that for the binomial closure $S^{bin}$ of any ring $S$ (or $\BZ_p$-algebra) discriminated by $\BZ_p$, the free pro-Hall $S^{bin}$-group $\FF(A,S^{bin})$ is fully residually free pro-$p$. Furthermore, we note that any finite set of elements in $\FF(A,S^{bin})$ explicitly defines a pro-$p$ subgroup of $\FF(A,S^{bin})$ and so an irreducible coordinate group over a free pro-$p$ group. 

In the particular case when the ring $S$ is universal for the class of finitely generated fully residually $\BZ_p$ rings (i.e. $S$ is fully residually $\BZ_p$ and any finitely generated  $\BZ_p$-algebra discriminated by $\BZ_p$ is a subring of $S$)  any finitely generated pro-$p$ subgroup of the free pro-Hall $S^{bin}$-group $\FF(A,S^{bin})$ is fully residually free pro-$p$ and we conjecture that the converse also holds, that is, any irreducible coordinate group over a free pro-$p$ group is a subgroup of $\FF(A,S^{bin})$.

\section{Basics on algebraic geometry over pro-$p$ groups}

Basics of algebraic geometry over arbitrary algebraic structures were established in \cite{DMR1}.  In this section we adjust these basic notions to the case of pro-$p$ groups and refer the reader to \cite{DMR1, DMR2} and \cite{Mel, Mel2}  for further details.

Let $\Gr(p)$ be the category of pro-$p$ groups. Objects in the category $\Gr(p)$ are groups which are inverse limits of an inverse system of discrete finite $p$-groups with the spectral topology and morphisms are continuous (in the spectral topology) homomorphisms of groups. Let $\FF=\FF(X)$ be the free object of this category, that is $\FF(X)$ is a free  pro-$p$ group with (finite) basis $X$.

We fix a pro-$p$ group $G$ and introduce the notion of a $G$-group. A pro-$p$ group $H$ is called a $G$-group if it contains a designated copy of $G$, which we identify with $G$. More precisely, a pro-$p$ group $H$  is a pro-$p$ group together with an embedding $\alpha: G\to H$. A morphism or $G$-homomorphism $\varphi$ from a $G$-pro-$p$ group $H_1$ to a  $G$-pro-$p$ group $H_2$ is a homomorphism of pro-$p$ groups which is the identity on $G$.

Obviously, $G$-pro-$p$ groups form a category that we denote by $\Gr_G(p)$.  Observe that the group $G[[X]]=G*_p \FF(X)$, where $*_p$ denotes the free product in the category $\Gr(p)$, is the free object (with basis $X=\{x_1,\dots,x_n\}$) in the category $\Gr_G(p)$.

A formal equality $f=1$, where $f\in G[[X]]$  can be treated  as an equation over $G$. Similarly, any subset $S \subset G[[X]]$ can be treated as a system of equations with coefficients in $G$. Note that in contrast with the case of abstract groups the left side of an equation is not necessarily a group word over $G$ in variables $X=\{x_1,\dots, x_n\}$.

Naturally, an element $g = (g_1,\ldots , g_n ) \in G^n $  is a solution of $f \in G[[X]]$ if
$$
f(g) = f(g_1 ,\ldots ,g_n) = 1.
$$
Similarly, a point $g \in G^n$ is a solution of a system $S$ if every polynomial from $S$ vanishes at $g$. We denote the algebraic set defined by $S$, that is the set of all solutions of $S$, by $V_G(S)$.

Let $g=(g_1,\dots, g_n)\in V_G(S)$. Denote by $\varphi _g$ the $G$-homomorphism from $G[[X]]$ to $G$ induced by the map $x_i\mapsto g_i$. Note that the set of solutions is in one-to-one correspondence with the set of homomorphisms $\varphi: G[[X]]\to G$ so that $S\subset \ker\varphi$, or equivalently, with the set of homomorphisms from $\factor{G[[X]]}{\ncl(S)}$ to $G$.

The set of all solutions of an equation $f=1$ is closed in the pro-$p$ topology on $G^n$ since it is the pre-image of the identity under the map $(g_1,\dots, g_n)\mapsto f(g_1,\dots, g_n)$ from $G^n$ to $G$. On the other hand, any algebraic set $V$ is also closed in the pro-$p$ topology since it is the intersection of closed sets corresponding to solutions of the equation.

The set $G^n$ can be endowed with the Zariski topology by taking all algebraic sets as a pre-basis of closed sets. The Zariski topology is weaker than the pro-$p$ topology. Recall that a Zariski-closed set is called irreducible if it can not be written as a union of two proper closed subsets.

We set 
$$ 
\rad(V_G(S)) = \mathop \bigcap \limits_{g \in V_G(S)} \ker \varphi _g.
$$
It is immediate to check that $\rad(V_G(S))$ is a normal subgroup of $G[[X]]$ and it is closed in the pro-$p$ topology as it is the intersection of kernels of the corresponding epimorphisms.  The factor group $\Gamma(S)=\Gamma(V_G(S))$ of $G[[X]]$ by the radical $\rad(V_G(S))$ is called the coordinate group associated to $S$ and is a pro-$p$ group. It follows immediately from the definition that the coordinate group $\Gamma(S)$ contains a copy of $G$ and is residually $G$. A coordinate group is called irreducible if so is $V$.

A universal $G$-sentence over $G$ is a sentence of the form:
$$
\forall x_1,\dots, x_n \Phi(X),
$$
where $\Phi(X)$ is a disjunction of conjunctions of finitely many equations and inequations of the form $v(X)=1$ or $u(X)\neq 1$, where $u,v\in G[[X]]$. Similarly, one defines existential $G$-sentences. Two pro-$p$ $G$-groups are called universally (existentially) equivalent if they satisfy the same set of universal (existential) $G$-sentences. Note that in contrast with the case of abstract groups, we enlarge the collection of terms and, in general, a term is not necessarily a group word over $G$ in variables $X$.

Recall that if $A$ and $B$ are two pro-$p$ $G$-groups, then $A$ is said to be \emph{fully residually $B$} or \emph{discriminated by $B$} if and only if for any  finite subset  of elements $C\subset A$ there exists a $G$-homomorphism $A\to B$ which is injective on $C$.

The following result is a reformulation of \cite[Theorem A, Theorem B and Remark 6.4]{DMR2} in our setting, see also \cite{Mel} and \cite[Section 1.5]{Mel2}. Its proof is analogous to the one given in \cite{DMR2} and we omit it.

\begin{thm}[see Theorem A, Theorem B and Remark 6.4 in \cite{DMR2}] \label{thm:un}  
For a finitely generated pro-$p$ $G$-group $H$ the following conditions are equivalent:
\begin{enumerate}
\item [1)] $H$ is $G$-discriminated by $G$;
\item [2)] $H$ is the  coordinate algebra of an irreducible
algebraic set over $G$.
\end{enumerate}
If $H$ satisfies either of the above conditions, then the following equivalent statements hold
\begin{enumerate}
\item [3)] $G\equiv_{\forall, G} H$;
\item [4)] $G \equiv_{\exists, G} H$;
\item [5)] $H$ $G$-embeds into the ultrapower of $G$.
\end{enumerate}
\end{thm}

Hence, one way to understand irreducible coordinate groups over  a pro-$p$ group $G$ is to study the class of pro-$p$ groups discriminated by $G$. 

\begin{rem}
Note that if the group $G$ is equationally Noetherian, i.e. the Zariski topology on $G^n$ is Noetherian for all $n$, then all the statements 1)-5) of Theorem \ref{thm:un} are equivalent. One can show, see \cite[Corollary 1.1]{Mel2}, that free nilpotent and metabelian pro-$p$ groups are equationally Noetherian. However, we do not know whether or not the free pro-$p$ group is equationally Noetherian.
\end{rem}

\section{Binomial Closure}
The main goal of this section is to prove that the binomial closure of any ring discriminated by $\BZ_p$ is fully residually $\BZ_p$. We first recall the definition of a binomial ring, see \cite{Hall,bincl}.
\begin{defn}
A ring $R$ is called \emph{binomial} if $R$ is a domain of characteristic $0$ and  for any element $\lambda \in R$ and any natural number $n$, the ring $R$ contains the following element (binomial coefficient):
\[
C_\lambda^n=\frac{\lambda(\lambda-1)(\lambda-2)\cdots (\lambda-n+1)}{n!}.
\]
Note that fields of characteristic zero, the ring of polynomials over such fields, the integers and the $p$-adic integers are binomial rings.
\end{defn}

\begin{defn} \label{defn:bcl}
Let $R$ be a domain of characteristic $0$. Consider the copy of $\BZ$ in $R$ generated by the unit of $R$ and let $\BZ^{-1}R=L$ be the localisation of $\BZ$ in $R$, see \cite{Bourb} for details. 

Define $R^{bin}$ recursively as follows. Let $R= R_0<L$, suppose that $R_i< L$ is already defined and define $R_{i+1}$ as follows
$$
R_{i+1}=\langle R_i, C_n^\alpha \mid \alpha \in R_i\smallsetminus R_{i-1}\rangle< L.
$$
Set $\varinjlim R_i=R^{bin}$
\end{defn}

\begin{lem}\label{lem:exthombin}\
\begin{enumerate}
\item  The binomial closure $R^{bin}$ is a binomial ring containing $R$;
\item  The binomial closure $R^{bin}$ satisfies the following universal property.  For any binomial ring $S$ and any homomorphism $\phi$ from $R$ to $S$ there exists a unique homomorphism $\phi':R^{bin}\to S$ which makes the following diagram commutative:
$$
\xymatrix@C3em{
R \ar@{^{(}->}[d]_{id} \ar[r]^{\phi}  & S\\
R^{bin} \ar[ru]_{\phi'}
}
$$
In addition, if $\phi$ is injective, then so is $\phi'$.
\end{enumerate}
\end{lem}
\begin{proof}
To see the first statement it suffices to note that for any $x\in R^{bin}$ there exists $i\in \BN$ so that $x\in R_i$. Hence, by definition, the binomial coefficients $C_x^n$ are elements of $R_{i+1}$ and so $R^{bin}$ is a binomial ring.

Without loss of generality, we assume that $\phi$ is non-trivial. Let $\bar S$ be the field of fractions of $S$ and consider the homomorphism $\phi'' = id \circ \phi: R \to \bar S$. Since $\phi''$ is a non-trivial homomorphism and so injective on $\BZ$, it follows that every element of $\BZ$ is sent to a non-trivial element and hence a unit of $\bar S$. Therefore, by the universal property of the localisation, we have that the homomorphism $\phi$ extends to a unique homomorphism $\bar \phi: L \to \bar S$. We have the following commutative diagram:
$$
\xymatrix@C5em{
R \ar@{^{(}->}[d]_{id} \ar[r]^{\phi}  & S \ar[d]^{id}\\
R^{bin} \ar@{^{(}->}[d]_{id} & S\ar@{^{(}->}[d]^{id} \\
L \ar[r]_{\bar\phi} &\bar S.
}
$$
We next show that the restriction $\phi'=\bar\phi |_{R^{bin}}$ from $R^{bin}$ to $\bar S$ has image in $S$, that is $\bar\phi |_{R^{bin}}(R^{bin})<S<\bar S$. 

For any element $x\in R^{bin}$ there exists $i\in \BN$ so that $x\in R_i$, $x\notin R_{i-1}$. By definition of $R_i$, it follows that for any such $x$ there exist a finite tuple of elements $\{C_{n_1}^{\beta_1}, \dots, C_{n_k}^{\beta_k}\}$, where $\beta_i\in R_{i-1}$  so that $x$ can be written as a polynomial expression in this set with coefficients in $R$.

Clearly,  $\bar\phi |_{R^{bin}}(R_0)< S$ since by definition $\bar\phi |_{R^{bin}}(R_0)= \phi(R)$. Assume by induction that for all $y\in R_j$, $j<i$, we have  that $\bar\phi |_{R^{bin}}(y)\in S$.  Since by induction $\bar\phi |_{R^{bin}}(\beta_m) \in S$ and $S$ is a binomial ring, it follows that $\bar\phi |_{R^{bin}}(C_{n_m}^{\beta_m})\in S$ for all $m=1,\dots, k$ and so $\bar\phi |_{R^{bin}}(x)\in S$, for all $x\in R_i$. Hence, we have obtained the following commutative diagram:
$$
\xymatrix@C5em{
R \ar@{^{(}->}[d]_{id} \ar[r]^{\phi}  & S \ar[d]^{id}\\
R^{bin} \ar@{^{(}->}[d]_{id} \ar[r]^{\bar\phi|_{R^{bin}}} & S\ar@{^{(}->}[d]^{id} \\
L \ar[r]_{\bar\phi} &\bar S.
}
$$
\end{proof}

Hence, for any domain $R$ of characteristic zero, the binomial closure of $R$ is a ring which is binomial, contains $R$ and is minimal with these properties, i.e. if $S$ is any binomial ring and $S<R$, then $R^{bin}<S$.

\begin{lem} \label{lem:discbin}
Let $R$ and $S$ be domains of characteristic zero, let $R^{bin}$ be the binomial closure of $R$ and suppose that $S$ is binomial. Suppose that $R$ is discriminated by $S$ by a family of homomorphisms $\{\phi_n\}$, then, in the notation of {\rm Lemma \ref{lem:exthombin}}, $R^{bin}$ is discriminated by $S$ by the family $\{\phi_n'\}$.
\end{lem}
\begin{proof}
Let $M\subset R^{bin}$ be a finite set of elements. By definition, $\varinjlim R_i=R^{bin}$ and we use induction on $i$. 

Firstly, note that for any finite $M$  there exists $R_i$ so that $M\subset R_i$. If $i=0$, then since $R=R_0$ is discriminated by $S$ by the family of homomorphisms $\{\phi_n=\phi_n'|_{R_0}\}$. This establishes the base of induction.  

Suppose by induction that for any finite set $M$ so that $M\subset R_j$, $j<i$ there is a homomorphism in the family $\{\phi_n'\}$  which is injective on $M$ and we prove the statement for $M\subset R_i$. 

Any element $x\in R_i$ can be written as a polynomial expression in $\{C_{n_1}^{\beta_1}, \dots, C_{n_k}^{\beta_k}\}$, where $\beta_i\in R_{i-1}$ and coefficients in $R$. Let $M'$ be obtained from $M$ by multiplying all elements of $M$ by a large enough integer $K$ so that $M'\subset R_{i-1}$. Note that since the characteristic of $R^{bin}$ is zero  we have that $Ks=Kt$ if and only if $s=t$, where $s, t \in R^{bin}$ are arbitrary. Hence,  all elements of $M'$ are pair-wise distinct and non-zero.

By induction, there exists a homomorphism $\phi'$ in the family $\{\phi_n'\}$ which is injective on $M'$ and maps its elements to pairwise distinct elements of $S$. We show that $\phi'$ is also injective on $M$. Indeed,  since $\phi'$ is a homomorphism on $R^{bin}$, for any element $M'\ni m'=Km$, where $m\in M$, we have  $\phi'(m')=\phi'(Km)=K\phi'(m)$. Since the characteristic of $S$ is zero, so $\phi'$ is injective on $M$.
\end{proof}

\section{Hall and Free Pro-Hall $R$-groups}
In this section we introduce free pro-Hall $R$-groups and show that if the binomial ring $R$ is fully residually $\BZ_p$, then the free pro-Hall $R$-group is fully residually free pro-$p$.

\begin{defn}
Let $R$ be a binomial ring. A Hall $R$-group $G$ is a nilpotent group together with a function $G\times R\to G$, $(g,r)\mapsto g^r$, $g\in G$, $r\in R$ satisfying the following universal axioms:
\begin{itemize}
\item
If $g\in G$ and $\alpha\in R$, let us denote by $g^\alpha$ the element $f_\alpha (g)$.
Then
\[
g^0=1, \quad g^{\alpha+\beta}=g^\alpha g^\beta, \quad (g^\alpha)^\beta=g^{\alpha\beta}.
\]
\item  For every $x,y\in G$ and for every $\lambda \in R$ we have $y^{-1}x^\lambda y=(y^{-1}xy)^\lambda$.
\item Let $c$ be the nilpotency class of $G$. Then for every $x_1,\dots, x_n\in G$ and for every $\lambda\in R$ the following axiom holds:
$x_1^\lambda\dots x_n^\lambda= (x_1\dots x_n)^\lambda \tau_2(X)^{C_\lambda^2}\cdots \tau_c(X)^{C_\lambda^m}$, where $X=\{x_1,\dots, x_n\}$ and $\tau_i(X)$ is the $i$-th Petresco word.
\end{itemize}
\end{defn}

Recall that for any $i$, the $i$-th Petresco word is defined in the free group $F$ on
$\{a_1,\dots, a_n\}$ recursively by the following formula:
\[
a_1^i\cdots a_n^i=\tau_1(A)^{C_i^1}\tau_2(A)^{C_i^2} \cdots \tau_i(A)^{C_i^i}.
\]

The axioms of Hall $R$-groups are identities, hence they define  a variety of nilpotent Hall $R$-groups of class $c$.  From general facts in universal algebra, see \cite{Mal}, it follows that the notions of $R$-subgroup, $R$-homomorphism, free Hall $R$-group, etc. are naturally defined. 

Let $F(A,R,c)$ be the free Hall $R$-group on $A=\{a_1,\dots, a_n\}$ of class $c$.  For all $c\ge 1$, there is a natural homomorphism $\psi_{c}:F(A,R,c+1)\to F(A,R,c)$. Given the inverse system $\{F(A,R,c), \psi_c\}$, we consider the inverse limit $\varprojlim F(A,R,c)$ which we denote by $\FF(A,R)$ and call the \emph{free pro-Hall $R$-group}.

It is well-known that the Hall basic commutators of weight less than $c+1$ form a basis of the free Hall $R$-group, i.e. every $g\in  F(A,R,c)$ can be uniquely written as follows:
$$
g=a_1^{\alpha_{11}}\cdots a_n^{\alpha_{1n}}\cdot \prod\limits_{i<j, i,j=1}^n [a_i,a_j]^{\alpha_{2,ij}}\cdots \prod\limits_{j=1}^{k_{l}} b_{j}^{\alpha_{l,j}}\cdots \prod\limits_{j=1}^{k_{c}} b_{j}^{\alpha_{c,j}},
$$
where $\alpha_{m,j}\in R$ and $b_j^{\alpha_{m,j}}$ are the Hall basic commutators of weight $m$.

It follows that the set of all Hall basic commutators form a basis of the free pro-Hall $R$-group $\FF(A,R)$, that is every element $g\in \FF(A,R)$ can be written as an infinite product:
\begin{equation}\label{eq:nf}
g=a_1^{\alpha_{11}}\cdots a_n^{\alpha_{1n}}\cdot \prod\limits_{i<j, i,j=1}^n [a_i,a_j]^{\alpha_{2,ij}}\cdot \prod\limits_{j=1}^{k_3} b_{j}^{\alpha_{3,j}}\cdots,
\end{equation}
where $\alpha_{m,j}\in R$ and $b_j^{\alpha_{m,j}}$ are the Hall basic commutators of weight $m$.

\begin{rem} \label{rem:hallprop}
Observe that if $R$ is the ring of $p$-adic integers $\BZ_p$, then any finitely generated torsion-free Hall $R$-group $G$ is a pro-$p$ group and the free Hall $R$-group $F(A,R,c)$ is the free nilpotent pro-$p$ group on $A$. Moreover, the pro-Hall $\BZ_p$-group $\FF(A,\BZ_p)$ is simply the free pro-$p$ group with base $A$. 

The above facts are well-known, but we were unable to find a reference. To prove these results, it suffices to observe that if $G$ is a finitely generated Hall $\BZ_p$-group, then the quotient map $\BZ_p\to \factor{\BZ}{p^n\BZ}$ gives rise to a quotient map from $G$ to a $p$-group $G_n$. Moreover, viewing  $\BZ_p$  as the inverse limit, we obtain the inverse limit of groups $\varprojlim G_n$. One can then check that $G\simeq  \varprojlim G_n$.
\end{rem}

Given an element $g\in \FF(A,R)$ in the form (\ref{eq:nf}) we define the $l$-truncation of $g$ to be 
\begin{equation}\label{eq:trunc}
g=a_1^{\alpha_{11}}\cdots a_n^{\alpha_{1n}}\cdots \prod\limits_{j=1}^{k_l} b_{j}^{\alpha_{l,j}}.
\end{equation}

We now show that if a $\BZ_p$-domain $R$ is $\BZ_p$-discriminated by $\BZ_p$, then the abstract group $\FF(A, R^{bin})$ is discriminated by the free pro-$p$ group $\FF$ with basis $A$.

\begin{thm} \label{thm:disc}
Let $R$ be an abstract ring {\rm(}or an abstract $\BZ_p$-ring{\rm)} and suppose that $R$ is fully residually $\BZ_p$.  Then the abstract group $\FF(A,R^{bin})$ is fully residually $\FF$.
\end{thm}
Note that if $R$ is a finitely generated ring from the universal class of $\BZ_p$, then since $\BZ_p$ is Noetherian, $R$ is fully residually $\BZ_p$, see \cite{DMR1}. 
\begin{proof}
By Lemma \ref{lem:exthombin}, every homomorphism $\phi:R\to \BZ_p$ induces a homomorphism $\phi':R^{bin}\to \BZ_p$. In turn, $\phi'$ induces a homomorphism $\tilde \phi:\FF(A,R^{bin}) \to \FF$ defined on elements of $\FF(A,R^{bin})$ written in the form (\ref{eq:nf}) as follows:
$$
\tilde \phi: \prod\limits_l\prod \limits_j b_{j}^{\alpha_{l,j}} \mapsto \prod\limits_l\prod \limits_j b_{j}^{\phi'(\alpha_{l,j})}.
$$

By Lemma \ref{lem:discbin}, $R^{bin}$ is discriminated by $\BZ_p$. For a finite set $M$ of elements of $\FF(A,R^{bin})$, there exists $K\in \BN$ so that the  $K$-truncations of all elements of $M$ are pair-wise distinct, see Equation (\ref{eq:trunc}). Let $N$ be the following finite set of elements of $R^{bin}$, $N=\{\alpha_{l,j}\mid l\le K\}$ and let $\psi$ be a homomorphism from $R^{bin}$ to $\BZ_p$ which is injective on $N$. 

From the above observation it follows that $\psi$ induces a homomorphism $\tilde \psi:\FF(A, R^{bin}) \to \FF$ which is injective on $M$.
\end{proof}

\begin{rem}
Note that even if the ring $R$ is universally equivalent to  $\BZ_p$ (in the usual first-order language of rings), the free pro-Hall $R$-group need not be a pro-$p$ group. Indeed, let $S$ be the Henselisation of $\BZ$ and let $S^{bin}$ be its binomial closure. Then, $S^{bin}$ is universally equivalent to $\BZ_p$ and so the free pro-Hall $S^{bin}$-group $\FF( A, S^{bin})$ is discriminated by $\FF$, but it is not a pro-$p$ group since all abelian subgroups of $\FF(A,S^{bin})$ are infinite but countable.
\end{rem}

\section{Finitely generated pro-$p$ groups in $\FF( A, R)$}

Let $R$ be a binomial $\BZ_p$-ring discriminated by  $\BZ_p$ considered as structure of the language of rings enriched by constants from $\BZ_p$. In this case $R=R^{bin}$ and $R$ has a natural structure of a $\BZ_p$-algebra. In this section we show that for any finite tuple of elements $\{h_1,\dots, h_m\}$ of the group $\FF(A)=\FF(A,R)$, there exists a finitely generated pro-$p$ group $H< \FF(A,R)$ so that $h_1,\dots, h_m\in H$.

Let $X=\{x_1,\dots, x_m\}$ and let $A=\{a_1,\dots, a_n\}$ be finite alphabets. Let $\FF=\FF(X)=\FF(X,\BZ_p)$ be the free pro-$p$ group and let $\FF(A)=\FF(A, R)$ be the free pro-Hall $R$-group.

Denote by  $\gamma_{c+1}(\FF(A))$ and  $\gamma_{c+1}(\FF(X))$ the $c+1$-th members of the lower central series of $\FF(A)$ and $\FF(X)$ correspondingly. Let $\FF(A)_c=\factor{\FF(A)}{\gamma_{c+1}(\FF(A))}$ and $\FF(X)_c=\factor{\FF(X)}{\gamma_{c+1}(\FF(X))}$.  

We fix an $m$-tuple of elements $(h_1(A), \dots, h_m(A))$ from the group $\FF(A)$ and let $h_i^c$ be the image of $h_i$ in $\FF(A)_c$ under the canonical projections. If $h_i$ is written in the form (\ref{eq:nf}):
$$
h_i= \prod\limits_l\prod \limits_j b_{j}^{\alpha_{l,j,i}}= \prod\limits_l h_{i,l},
$$
then $h_i^c$ is the $c$-truncation of $h_i$.

We define a map $\theta_c: \FF(X)_c\to \FF(A)_c$. Since for any $c\ge 1$ the group $\FF(X)_c$ is $\BZ_p$-generated by $X=\{x_1^c,\dots, x_m^c\}$, it suffices to define $\theta_c$ on the $\BZ_p$-generators of $\FF(X)_c$. Set $\theta_c: x_i^c\mapsto h_i^c$.

Let now $\langle h_1^c, \dots, h_m^c\rangle_R$ be the Hall $R$- subgroup of $\FF(A)_c$ generated by the indicated set. Since $R$ is a $\BZ_p$-algebra, $\langle h_1^c, \dots, h_m^c\rangle_R$  contains a Hall $\BZ_p$-subgroup $H_c=\langle h_1^c, \dots, h_m^c\rangle_{\BZ_p}$  in a natural way. Since $H_c$ is a torsion-free Hall $\BZ_p$-group, by Remark \ref{rem:hallprop}, we have that it is a pro-$p$ group.
 
Since $\FF(X)_c$ is a free Hall $\BZ_p$-group, so $\theta_c:\FF(X)_c\to H_c\hookrightarrow \FF(A)_c$ is a well-defined homomorphism, it can be explicitly described by the following formula
$$
\theta_c:  w^c(X)\mapsto w^c(h_1^c,\dots, h_m^c),
$$
where $w^c\in \FF(X)$. 

Since $h_i^c$ is simply the $c$-truncation of $h_i$, it is clear that there is a natural homomorphism $\psi_c^{c+1}:H_{c+1}\to H_c$ so that $\psi_c^{c+1}(\theta_{c+1}(w^{c+1}(X)))=\theta_c(w^c(X))$.

Thus, we obtain the inverse spectrum $\langle H_c;  \psi_c^{c+1}(H_{c+1})=H_c\rangle$ of pro-$p$ groups, where $c\in \BN$.  Let $H$ be the inverse limit of this spectrum. Taking the limit of the maps $\theta_c$, we obtain a well-defined homomorphism $\theta: \FF(X)\to \FF(A)$ described by the formula:
$$
\theta: w(x_1,\dots, x_m)\mapsto w(h_1,\dots, h_m).
$$
We call the homomorphism $\theta$ the \emph{substitution map}. The group $\theta (\FF(X))<\FF(A, R)$ is isomorphic to $H$ and is a pro-$p$ group.

\begin{defn}
We call the pro-$p$ group $H$ constructed above \emph{the pro-$p$ subgroup of $\FF(A, R)$ generated by the elements $\{h_1,\dots, h_m\}$} and denote it by $H=\langle h_1,\dots, h_m\rangle_{\FF}$.
\end{defn}

We summarise the discussion above in the following 
\begin{prop}
The substitution homomorphism $\theta:\FF(X)\to \im(\theta)\simeq H <\FF(A, R)$ is a continuous map.
\end{prop}

\begin{thm}
Let $H$ be a finitely generated pro-$p$ subgroup of $\FF(A, R)$, $H=\langle h_1,\dots, h_m\rangle_{\FF}$, then $H$ is fully residually $\FF$ and hence is a coordinate group of an irreducible algebraic set over $\FF$.
\end{thm}
\begin{proof}
From Theorem \ref{thm:disc}, it follows that the abstract group $H$ is fully residually $\FF$. Since $\FF$ is the free pro-$p$ group and the pro-$p$ topology of the finitely generated pro-$p$ group $H$ is uniquely determined by the group structure (see \cite[Theorem 4.3.5]{Wil}), it follows that the discriminating family is in fact continuous.
\end{proof}

\begin{rem}
As we mentioned in the introduction, for the class $\mathcal L$  of finitely generated pro-$p$ subgroups of sequences of extensions of centralisers of the free pro-$p$ group  introduced in \cite{KZ}, it is not known whether  $\mathcal L$ contains only fully residually free pro-$p$ groups. In fact, there is a natural epimorphism $\pi$ from extensions of centralisers of the free pro-$p$ group to pro-$p$ subgroups of the free pro-Hall $\BZ_p[t]^{bin}$-group. For example,
$$
H=\langle \FF(a,b), s \mid [s, C(w(a,b))]=1 \rangle \to \langle a,b, w(a,b)^t \rangle_{\FF} < \FF(\{a,b\}, {\BZ_p[t]}^{bin}).
$$
Moreover, iterated sequences of extensions of centralisers are fully residually $\FF$ if and only if the natural epimorphism $\pi$ is an isomorphism. Indeed, the family of homomorphisms $\phi_\alpha: \BZ_p[t]\to \BZ_p$ induced by the maps $t\mapsto \alpha$ is discriminating, here $\alpha$ runs over non-trivial elements of $\BZ_p$. By Theorem \ref{thm:disc}, the family $\{\phi_\alpha\}$ induces a discriminating family $\{\tilde \phi_\alpha\}$ from the group $\FF(X, \BZ_p[t]^{bin})$ to $\FF(X, \BZ_p)$ that maps $w^t$ to $w^\alpha$. On the other hand, any homomorphism $\psi$ from $H$ to $\FF$ satisfies that $\psi(s)\in C(w)=\langle w^\alpha\mid \alpha\in \BZ_p\rangle$. Therefore,  $\psi$ extends to a homomorphism $\tilde \phi_\alpha$ for some $\alpha\in \BZ_p$ and the following diagram is commutative:
$$
\xymatrix@C3em{
H \ar[d]_{\psi} \ar@{->>}[r]  &  \langle a,b, w(a,b)^t \rangle_{\FF}\ar[dl]^{\tilde \phi_\alpha}\\
\FF(\{a,b\}, \BZ_p)
}
$$
In particular,  if $H$ is fully residually $\FF(\{a,b\}, \BZ_p)$, then by definition for every $h\in H$ there exists $\psi:H\to \FF(\{a,b\}, \BZ_p)$ injective on $h$ and by the commutativity of the diagram it follows that $\pi(h)$ is a non-trivial element of $\langle a,b, w(a,b)^t \rangle_{\FF}$. We conclude that if $H$ is fully residually $\FF(\{a,b\}, \BZ_p)$, then $\pi$ is also injective and so an isomorphism.
\end{rem}

Every finitely generated pro-$p$ subgroup of $\FF(A, R^{bin})$, where $R$ belongs to the universal class of $\BZ_p$ (in the language with constants), is discriminated by the free pro-$p$ group. We believe that the converse also holds and formulate the following conjecture.

\begin{conj}\label{conj:main}\
\begin{enumerate}
\item
Let $G$ be a finitely generated fully residually $\FF$ pro-$p$ group, then there exists a finitely generated binomial $\BZ_p$-algebra $R=R(G)$ from the universal class of $\BZ_p$ so that $G$ embeds into the free pro-Hall $R$-group $\FF(A, R)$.
\item
The universal class of $\BZ_p$ contains a $\BZ_p$-ring $S$ so that any finitely generated fully residually $\FF$ pro-$p$ group embeds into the free pro-Hall $S$-group $\FF(A, S)$. The ring $S$ is a universal ring for the class of finitely generated $\BZ_p$-rings from the universal class of $\BZ_p$, that is $S$ is a $\BZ_p$-ring, it is fully residually $\BZ_p$ and every finitely generated $\BZ_p$-ring from the universal class of $\BZ_p$ is a subring of $S$.
\end{enumerate}
\end{conj}

\end{document}